\documentclass[12pt,oneside,reqno]{amsart}

\usepackage[numbers,sort&compress]{natbib}
\usepackage{color}
\usepackage{graphicx}
\usepackage{amsthm}
\usepackage{mathdots}
\usepackage{soul}
\usepackage{physics}
\usepackage{enumitem}
\usepackage{tikz}
\usepackage{amssymb,amsthm,amsmath, subcaption}
\usepackage[numbers,sort&compress]{natbib}
\usepackage[backref=page, colorlinks=true, allcolors=blue]{hyperref}
\usepackage{color}

\usepackage{graphicx}
\usepackage{tikz}
\usetikzlibrary{shapes}
\usepackage{xparse}
\usepackage{array}
\usetikzlibrary{calc,arrows.meta,positioning}

\expandafter\def\expandafter\normalsize\expandafter{%
    \normalsize%
    \setlength\abovedisplayskip{3pt}%
    \setlength\belowdisplayskip{3pt}%
    \setlength\abovedisplayshortskip{3pt}%
    \setlength\belowdisplayshortskip{3pt}%
}

\title[Algebraic Properties of the Ideal of Spectral Invariants]{Algebraic Properties of the Ideal of Spectral Invariants for the Discrete Laplacian}
\author[M. Faust]{Matthew Faust}
\address{Matthew Faust, Department of Mathematics, Michigan State University, East Lansing, MI, USA} \email{mfaust@msu.edu}
\author[L. Friedman]{Leo Friedman}
\address{Leo Friedman, Department of Mathematics and Statistics, Pomona College, Claremont, CA, USA} \email{ldfq2023@mymail.pomona.edu}
\author[G. O'Malley]{Gavin O'Malley}
\address{Gavin O'Malley, Department of Mathematics, Michigan State University, East Lansing, MI, USA} \email{omalle93@msu.edu}
\author[R. Ramos]{Rolando Ramos}
\address{Rolando Ramos, Department of Mathematics, Michigan State University, East Lansing, MI, USA} \email{ramosrol@msu.edu}
\author[A. Sharma]{Aaryan Sharma}
\address{Aaryan Sharma,  Department of Mathematics, Texas A\&M University, College Station, Texas,  USA} \email{aaryans16@tamu.edu}

\newcounter{Problem}

\theoremstyle{plain}
\newtheorem{theorem}{Theorem}[section]
\newcommand{\R}{\mathbb{R}}
\newtheorem{corollary}[theorem]{Corollary}
\newtheorem{lemma}[theorem]{Lemma}

\newtheorem{remark}{Remark}
\newtheorem{problem}[Problem]{Problem}

\newcommand{\CC}{\mathbb{C}}
\newcommand{\TT}{\mathbb{T}}
\newcommand{\ZZ}{\mathbb{Z}}

\newcommand{\bfz}{{\bf 0}}
\newcommand{\bq}{{\bf q}}
\theoremstyle{plain}

\newtheorem{conjecture}{Conjecture}

\newcommand{\defcolor}[1]{{\color{blue}#1}}
\newcommand{\demph}[1]{\defcolor{{\sl #1}}}

\begin{document}

	\begin{abstract}
	Let  $\Gamma=q_1\mathbb{Z}\oplus q_2 \mathbb{Z}\oplus\cdots\oplus q_d\mathbb{Z}$, with $q_j\in \mathbb{Z}^+$ for each $j\in \{1,\ldots,d\}$, and denote by $\Delta$ the discrete Laplacian on $\ell^2\left( \mathbb{Z}^d\right)$. We describe various algebraic properties of the ideal of spectral invariants for the discrete Laplacian when $d=1$, including a construction of a Gr\"obner basis. We also present various collections of complex $\Gamma$-periodic potentials $V$ that are such that $\Delta$ and $\Delta + V$ are Floquet isospectral. We end with a discussion of the general setting, where the $q_i$ are taken to be vectors in $\ZZ^d$. 
    \end{abstract}
	
	\maketitle 
	\section{Introduction}~\label{Sec:Intro}
    The discrete Laplacian $\Delta$ acts on square-summable functions $f$ on $\ZZ^d$, $\ell^2(\ZZ^d)$, via:
  \[\left(\Delta f\right)(n)=\sum_{\|n-m\|_1 = 1}f(m),\]  
\noindent where for $n=(n_1,n_2,\dots,n_d), m=(m_1,m_2,\dots,m_d)\in \ZZ^d,$
    $\|n-m\|_1:= \sum _{j=1}^{d}|n_j - m_j|.$
    For $\Gamma = q_1\ZZ~\oplus~q_2\ZZ~\oplus\dots\oplus q_d\ZZ$, a $\Gamma$-periodic potential $V:\ZZ^d \rightarrow \CC$ is a function where
    \[ V(n)=V(n+\gamma) \text{ for all } n^d\in \ZZ \text{ and }\gamma \in \Gamma.\]
    We call the sum $\Delta + V$ a \demph{discrete periodic Schrödinger operator}. 
    Floquet theory reveals that the $\ell^2(\ZZ^d)$ spectrum of this operator, $\sigma(\Delta + V)$, is given by solutions to the equation 
    $$(\Delta f)(n)+V(n)f(n)=\lambda f(n),$$
    for functions $f:\ZZ^d \to \CC$ with Floquet-Bloch boundary conditions
    $$f(n+q_j\mathbf{e}_j) = e^{2 \pi ik_j} f(n),$$ with $n\in \ZZ^d, k_j\in[0,1),$ and $j = 1,2,\dots, d.$ Here, $\{\mathbf{e}_j\}_{j=1}^{d}$ represents the standard bases in $\mathbb{R}^d$ and $\{q_j\}_{j=1}^d$ are positive integers.
    
    Noticing that each $f(n)$ is determined by its values on $\{ (n_1,\dots, n_d) \mid  0 \leq n_i < q_i \}$, for each fixed $k \in [0,1)^d$, $\Delta+V$ acts as a multiplication by a finite matrix. Treating $z_i = e^{2 \pi ik_j}$ as indeterminates, we let $L_V(z)$ with $z=(z_1,z_2,\dots,z_d)$ denote this finite \demph{Floquet matrix}. In particular, letting $\TT$ denote the complex unit circle, one has \[ \sigma(\Delta + V) = \bigcup_{z \in \TT^d} \sigma(L_V(z)).\]  We say that two discrete periodic Schr\"odinger operators $V+\Delta$ and $V' +\Delta$ with $\Gamma$-periodic potentials (or the $\Gamma$-periodic functions $V$ and $V'$ themselves) are Floquet isospectral if
    \[\sigma(L_V(z))=\sigma(L_{V'}(z))\text{, for all } z\in \mathbb{\TT}^d.\]
    
    The study of Floquet isospectrality dates back to at least the late 1980s, where Kappeler released three foundational papers~\cite{Kap3, Kap2, Kap1}. These works addressed many properties including, but not limited to, how many generic values of the quasimomenta need to be tested to determine the Bloch variety~\cite{Kap2}, how many potentials can be Floquet isospectral to any given potential~\cite{Kap3} (and what these must be in some settings~\cite{Kap2}), separability~\cite{Kap1}, and, as lies our interest, Floquet isospectrality for the discrete Laplacian~\cite{Kap1}. 
    
    In~\cite{Kap1}, it was established that the only real potential such that $\Delta$ and $\Delta+V$ are Floquet isospectral is the zero potential $V=0$. This result was generalized in~\cite{Saburova2024} to apply to a large class of graphs. See~\cite{Burak} for a related problem.
    There has also been great progress in studying Floquet isospectrality for separable potentials~\cite{liujde}. 
    
    Floquet isospectrality is one of many problems in the spectral theory of discrete periodic operators where algebro-geometric methods have played an important role. Besides the more general Fermi isospectrality~\cite{Liu_2022, liu2021fermi, chu2025weakseparabilitypartialfermi}, algebraic methods have found their way into the study of the (ir)reducibility of Bloch and Fermi varieties~\cite{flm22, flm23, fg25, Shipman, LiShip, LeeLiShip}, the existence of flat bands~\cite{faust2025rareflatbandsperiodic, faust2025absenceflatbandsdiscrete, sabri2023flat, li2025eigenvaluesmaximalabeliancovers, spier2025eigenvaluesuniversalcoversmatching, Kuch1989}, Borg's theorem~\cite{liuborg}, quantum ergodicity~\cite{ms22,liu2022bloch}, the study of extrema~\cite{FS, faust2025spectraledgesconjecturecorners, fk2, alon2025smoothcriticalpointseigenvalues, FLL2025, dks, faust2025criticalpointdegreeperiodic, LG23}, and more; see \cite{abdulrahman2025sharppolynomialvelocitydecay, fillman2025measurespectraperiodicgraph, Embed, Overview, liujmp22, shipmansottile, kuchment2023analytic, BerkKuch} for further developments and surveys.

    This paper studies the Ambarzumyan-type inverse problem of finding $\Gamma$-periodic potentials Floquet isospectral to the zero potential (\textbf{0}). Although it is well-known that there are no nonzero real-valued potentials Floquet isospectral to \textbf{0}, nonzero complex-valued potentials may exist. Showing the existence of such potentials has become a recent topic of interest~\cite{SabMartHex, flmrp, CFK}.  In~\cite{flmrp}, it was shown that there exist explicit complex nonzero $\Gamma$-periodic potentials that are Floquet isospectral to $\bfz$ when at least one of the $q_j$ is an even number greater than $3$. Recently, in~\cite{CFK} it was shown that there exist nonzero $\Gamma$-periodic potentials Floquet isospectral to $\bfz$ when at least one $q_j > 3$.  Although the methods of~\cite{CFK} prove the existence of these potentials, unlike~\cite{flmrp}, they do not provide information on what these potentials are. Towards this end, we offer several general forms for which we suspect that a class of nonzero complex potentials Floquet isospectral to $\bfz$ will take. We also present some theoretic results and discuss various open problems and conjectures. 
    
   The remainder of this paper is organized as follows. In Section~\ref{Sec:2}, we introduce our main object of study, the ideal of spectral invariants. In Section~\ref{Sec:3}, we construct a Gr\"obner basis for the ideal of spectral invariants and discuss some of its consequences. In Section~\ref{Sec:4}, we introduce a specialization of the potential that gives rise to a system with roughly half the variables and polynomials, and we also describe various symmetries and properties that potentials Floquet isospectral to $\bfz$ must satisfy given their existence. In Section~\ref{Sec:5}, we present various collections of experimental results found using Bertini and Macaulay2. Finally, in Section~\ref{Sec:6}, we discuss when the period is allowed to take on more general full rank sub-lattices of $\ZZ^d$, where some rigidity properties of~\cite{CFK} no longer hold.

\section{Spectral Invariants}$~$ ~\label{Sec:2}
As $n\ZZ$-periodic potentials isospectral to $\bfz$ naturally extend to $q_1\ZZ \oplus q_2\ZZ \dots q_d \ZZ$-periodic potentials isospectral to $\bfz$ when $q_k = n$ for some $k$, we will focus on the one-dimensional case. Observe that any $n\mathbb{Z}$-periodic function on the vertices of the one-dimensional square lattice is uniquely determined by its values on the vertices $1,\dots,n$, and thus any such function can be represented by a vector of these $n$ values.  Given a $n\mathbb{Z}$-periodic potential $V = (v_1, v_2, \ldots, v_n)$ with $n \geq 3$, the Floquet matrix $L_V(z)$ has the following form:

$$L_V(z) = \begin{pmatrix}
    v_1 & 1 & & z^{-1} \\
    1 & v_2 & \ddots & \\
     & \ddots &\ddots & 1 \\
    z &  & 1 & v_n
\end{pmatrix}.$$

 It is well-known in this setting (see ~\cite[page 1]{Kap2}) that the spectrum of $\sigma(L_V(z))$ for each $z \in \TT$ is determined by the spectrum of $L_V(z)$ at a single value of $z \in \TT$. Thus, from here on we write $L_V$ in place of $L_V(1)$, and note that $V$ and $\bfz$ are Floquet isospectral if $\sigma(L_V) = \sigma(L_{\bfz})$ (including multiplicity of eigenvalues). 

Let $D_V(\lambda)$ denote the characteristic polynomial of $L_V$. Notice that in order for $V$ and $\bfz$ to be Floquet isospectral, $D_V(\lambda)$ and $D_{\bfz}(\lambda)$ must be the same univariate polynomials in $\lambda$. Let $\zeta_V$ denote the polynomial given by the difference of $D_V(\lambda)$ and $D_{\bfz}(\lambda)$, that is $$\zeta_V:=D_V(\lambda)-D_{\bfz}(\lambda).$$

Let $[\lambda^k]f$ denote the coefficient of the term with monomial $\lambda^k$ of a polynomial $f$. In order for $D_V(\lambda)$ and $D_{\bfz}(\lambda)$ to correspond to the same polynomials their coefficients must agree. If this holds, then, noting that $[\lambda^n] \zeta_V=0$ and treating the $v_j$ as indeterminates, \begin{equation}~\label{eq:sysSpecInvars} \demph{(p_k := )}(-1)^{n-k}[\lambda^{n-k}]\zeta_V = 0 \text{ for } k \in \{1, \dots, n\} =: [n].\end{equation}

As the $p_k$ are polynomials in the $v_j$, ~\eqref{eq:sysSpecInvars} defines a system of $n$ polynomial equations in $n$ variables. As $v_j = 0$ for all $j \in [n]$ is trivially a solution to this system, the constant terms the $p_k$ are zero. We call the $p_k$ \demph{spectral invariants} (for the potential $\bfz$). 

Let $e_k$ be the elementary symmetric polynomial of degree $k$ in $\CC[v] = \CC[v_1,\dots, v_n]$, $$e_k = \sum_{1 \leq j_1 < \cdots < j_k \leq n} \prod_{\ell = 1}^k v_{j_\ell},$$ define $T_d(f)$ to be the homogeneous polynomial given by the degree $d$ terms of $f$, and let $T_{\max}(f) = T_{\deg(f)}(f)$.

As the coefficients of $D_{\bfz}(\lambda)$ are constant in $\CC[v]$, for a nonconstant monomial $m \in \CC[v]$, \begin{equation}~\label{specinvars} [m]p_k = [\lambda^{n-k}m] (-1)^{n-k} D_V(\lambda).\end{equation} Recall that, \begin{equation}~\label{eq:leibnizExpan} D_V(\lambda) = \sum_{\sigma \in S_n} \text{sgn}(\sigma) \prod_{j=1}^n (L_V - \lambda I)_{j,\sigma(j)},\end{equation} where $\sigma \in S_n$, the symmetric group on $[n]$.  Notice that the maximum degree terms of $\CC[v]$ for each fixed $\lambda^{n-k}$ occur in ~\eqref{eq:leibnizExpan} for the identity permutation. Thus, \begin{equation}~\label{zetalambdak} T_{\max}(p_{n-k}) = [\lambda^{k}]\left( (-1)^{k} \prod_{j=1}^n (v_j - \lambda)\right) = e_{n-k}.\end{equation}

Although the leading homogeneous terms of the $p_k$ are symmetric, this is not true for $p_k$ when $n>3$. Instead, the $p_k$ are fixed by the dihedral group, $D_n$, that is, \begin{equation}~\label{dihedral} p_k(v_1,v_2,\dots, v_n) = p_k(v_{\sigma(1)}, v_{\sigma(2)}, \dots, v_{\sigma(n)}) \text{ for all } \sigma \in D_n. \end{equation} 
This relationship is easily derived from the matrix $L_V$. 
\smallskip
\begin{lemma}~\label{parity}
    If $n-k$ is odd, then $p_{n-k}$ contains only odd degree terms and if $n-k$ is even, then $p_{n-k}$ contains only even degree terms.
\end{lemma}

\begin{proof}
     By~\cite[Section 2 (ix)]{flmrp}, the only $\sigma \in S_n$ such that $\displaystyle \prod_{j=1}^n (L_V - \lambda I)_{j,\sigma(j)} \neq 0$ are those made up of only $1$-cycles and $2$-cycles. Let $m$ be some monomial of degree $s>0$ in $\CC[v]$. By ~\eqref{specinvars} and ~\eqref{eq:leibnizExpan}, we have \[ [m] p_{n-k} = [m\lambda^k]\sum_{\sigma \in S_n} \text{sgn}(\sigma) \prod_{j=1}^n (L_V - \lambda I)_{j,\sigma(j)}. \]
    Notice that for a fixed $\sigma \in S_n$, for $\displaystyle [m\lambda^k] \sum_{j=1}^n (L_V - \lambda I)_{j,\sigma(j)}$ to be nonzero, $\sigma$ must have exactly $s+k$ fixed points. As the remaining cycles of $\sigma$ must be $2$ cycles, it follows that $n-k-s$ must be even. Thus, the parity of $n-k$ and $s$ must agree.
\end{proof}

By ~\eqref{zetalambdak} and Lemma~\ref{parity}, it follows that $p_k = e_k + f_k$ where $f_k$ is some polynomial in $\R[v_1,\dots,v_n]$ with $\deg(f_k) < k-1$. 

We next turn to the \demph{ideal of spectral invariants} (for the potential $\bfz$). That is, the ideal generated by our spectral invariants, $$I := \langle p_1, p_2, \ldots, p_n \rangle.$$ The vanishing set, $V(I)$, contains exactly the solutions to ~\eqref{eq:sysSpecInvars}. 

\section{A Gröbner Basis for The Ideal of Spectral Invariants}$~$ ~\label{Sec:3}
For an ideal $I$, a Gröbner basis $G = \{g_1, g_2, \ldots, g_k\}$ is a minimal set of polynomial generators such that 
$$I = \langle g_1, g_2, \ldots, g_k \rangle \text{ and }LT(I) = \langle LT(g_1), LT(g_2), \ldots, LT(g_k) \rangle,$$
where $LT(g_j)$ is the leading term of $g_j$ and $LT(I)$ is the set of all leading terms of polynomials in $I$ with respect to some global monomial ordering.
We will be using the graded reverse lexicographic (grevlex) monomial ordering on $\CC[v]$, defined as follows. For $\alpha = (\alpha_1, \ldots, \alpha_n) \in \mathbb{Z}_{\geq 0}^n$, define $v^\alpha = v_1^{\alpha_1}v_2^{\alpha_2}\cdots v_n^{\alpha_n}$. Then, in grevlex ordering,  $v^\alpha > v^\beta$ if $\deg(v^\alpha) > \deg(v^\beta)$ or if $\deg(v^\alpha) = \deg(v^\beta)$ and the rightmost nonzero entry in $\alpha - \beta$ is negative. 
See~\cite[Chapter 2]{CLOIVA} for more details on Gr\"obner bases and monomial orderings.

For $a\in\{1,\dots,n\}$ and $b\in\{0,\dots,n\}$, define the polynomials \begin{equation}~\label{eq:comphom}H(a, b) := \sum_{a \leq j_1 \leq \cdots \leq j_b \leq n} \prod_{\ell = 1}^b v_{j_\ell} \text{, and let } H(a,0) := 1.\end{equation}
These are exactly the complete homogeneous symmetric polynomials of degree $b$ on the variables $v_a,\dots, v_n$.
For $k=1,\dots,n$, define 
\begin{equation} g_k := -\sum_{j=1}^k H(k, k-j)(-1)^jp_j.\end{equation}

\begin{theorem}~\label{grobnerbasis}
    $G = \{g_1, \ldots, g_n\}$ is a Gröbner basis for $I$.
\end{theorem} 

The remainder of this section will be devoted to proving Theorem~\ref{grobnerbasis} and discussing some of its consequences. 

\begin{lemma}~\label{equalideals} 
    $\langle g_1, \ldots, g_n\rangle = I$. 
\end{lemma}
\begin{proof}
    Note that $\langle g_1, \ldots, g_n \rangle \subseteq I$ follows directly from the definition of $g_k$.
    
    We will show by induction that each $p_j \in \langle g_1, \ldots, g_n \rangle$. Observe that 
    $$g_1 = -\sum_{j=1}^1H(1, 1-j)(-1)^jp_1 = H(1, 0)p_1 = p_1.$$
    
    Thus, $p_1 \in \langle g_1, \ldots, g_n \rangle$. Now suppose that for each $j < k$, $p_j \in \langle g_1, \ldots, g_n \rangle$. Then, 
    $$p_j = \sum_{\ell=1}^n P_{j,\ell}g_\ell,$$  
    where each $P_{j,\ell} \in \mathbb{C}[v_1, \ldots, v_n]$. By the definition of $g_k$,
    $$g_k = -\sum_{j=1}^{k-1} H(k, k-j)(-1)^jp_j - (-1)^kp_k.$$
    It follows that each $p_k \in \langle g_1 \ldots, g_n \rangle$, and so $I \subseteq \langle g_1, \ldots, g_n \rangle$.  
\end{proof}
\smallskip
\begin{theorem}~\label{homogeneous}
    $T_{\max}(g_k) = H(k, k)$.
\end{theorem}

\begin{proof}
    Since we are only concerned with the highest degree portion of $g_k$, we can ignore any term of less than maximal degree. From ~\eqref{zetalambdak}, we have $T_{\max}(p_j) = e_j$, so, unless all degree $k$ terms cancel, we can ignore any lower order terms of $p_j$. We will proceed assuming that $g_k$ is degree $k$ (and prove this assumption). By our assumption, we have
    
    $$T_{\max}(g_k) = T_k(g_k) = -\sum_{j=1}^k H(k, k-j) (-1)^je_j.$$

    Notice that $\displaystyle \sum_{j=0}^n (-1)^j e_j x^j = \prod_{i=1}^n  (1-v_i x)$ and $\displaystyle \sum_{j=0}^\infty H(k,j) x^j = \prod_{j=k}^n \frac{1}{1-v_j x}$.

    Thus, \[\left[x^k\right] \left(\sum_{j=0}^n (-1)^j e_j x^j\right) \left(\sum_{j=0}^\infty H(k,j) x^j\right) = H(k,k)-T_k(g_k).\]
    However, we have \[\left(\sum_{j=0}^n (-1)^j e_j x^j\right) \left(\sum_{j=0}^\infty H(k,j) x^j\right) = \left( \prod_{i=1}^n  (1-v_i x)\right) \left( \prod_{j=k}^n \frac{1}{1-v_j x} \right)\] \[= (1-v_1 x)\cdots (1-v_{k-1} x),\] and therefore \[\left[x^k\right] \left(\sum_{j=0}^n (-1)^j e_j x^j\right) \left(\sum_{j=0}^\infty H(k,j) x^j\right) = 0.\]
 
   Thus, we find
 $T_k(g_k)  = H(k, k)$. 
\end{proof}
\smallskip
\begin{corollary}~\label{leadingterm}
    $LT(g_k) = v_k^k$ in grevlex monomial ordering.
\end{corollary}

Given two polynomials $f$ and $g$, let $LM(f)$ denote the monomial of the term $LT(f)$, let $LCM(f,g)$ denote the least common multiple, and let $S(f,g)$ denote the $S$-polynomial of $f$ and $g$, that is $\displaystyle S(f,g) = \frac{LCM(LT(f),LT(g))f}{LT(f)} - \frac{LCM(LT(f),LT(g))g}{LT(g)}$. Finally, for a set of polynomials $K$, we write $f \to_K a$ to mean that $a$ is the remainder of $f$ after multivariate division by the elements of $K$. We are now ready to prove Theorem~\ref{grobnerbasis}.
\begin{proof}[Proof of Theorem~\ref{grobnerbasis}]
    By Lemma~\ref{equalideals} we have that $\langle g_1, \ldots, g_n\rangle = I$. By \cite{CLOIVA}[Chapter 2, Section 10, Proposition 1], if $LM(g_j)$ and $LM(g_k)$ are relatively prime, then $S(g_j, g_k) \to_G 0$. Since $LT(g_k) = v_k^k$, each pair $g_j$ and $g_k$, $j \neq k$, are clearly relatively prime. Therefore, by Buchberger's Criterion~\cite[Chapter 1]{CLOII} $G$ is a Gröbner basis of $I$.
\end{proof}
\smallskip
\begin{theorem}~\label{hilbert}
    The affine Hilbert polynomial of $I$ is $HP_{R/I}(s) = n!$. 
\end{theorem}

\begin{proof} 
    Let $R = \CC[v]$. For $s \geq 0$, the affine Hilbert function of I, $HF_{R/I}(s)$, is the number of monomials of total degree $\leq s$ not in $LT(I)$. For $s$ sufficiently large, $HF_{R/I}(s)$ agrees with the affine Hilbert polynomial $HP_{R/I}(s)$. Since $G$ is a Gröbner basis for $I$, \[\langle LT(I) \rangle = \langle LT(g_1), \ldots, LT(g_n) \rangle = \langle v_1, v_2^2, \ldots, v_n^n \rangle.\] Thus, the set of monomials not in $LT(I)$ are exactly
    
    $$B = \left\{v_1^{j_1} v_2^{j_2} \cdots v_n^{j_n} | 0 \leq j_\ell < \ell\right\}.$$
     Each $j_\ell$ has $\ell$ values it can take, and so there are $n!$ monomials in the set. 
     
     As all monomials in $B$ are at most degree $\binom{n}{2}$, for any $s\geq \binom{n}{2}$, $HF_{R/I}(s) = n!$. It follows that $HP_{R/I}(s) = n!$. 
\end{proof}
\smallskip
\begin{corollary}~\label{Cor:GBConsequences}  The dimension of $V(I)$ is $0$, and the degree of $V(I)$ is $n!$.
 Consequently, $V(I)$ has $n!$ points counting multiplicity and is composed of between $1$ and $n!$ points. \footnote{Through different methods, this was also proven in~\cite{Kap3} and in ~\cite{VPapan}, although both can be seen as special cases of a similar result for additive inverse eigenvalue problems which was first proven in \cite{FRIEDLAND197715}.}
\end{corollary}
\smallskip
\begin{remark}~\label{Rm:gen}
    Notice that the proofs of Theorems~\ref{grobnerbasis} and~\ref{hilbert} only relied on the highest degree terms of the $p_k$ being the elementary symmetric polynomials of the degree $k$; their lower degree parts $f_k$ are just required to be less than degree $k$. That is, this procedure gives a Gr\"obner basis for any ideal generated by perturbations of elementary symmetric polynomials, including any ideal of spectral invariants for any finite graph.
\end{remark}

We finish by improving the upper bound on the number of distinct points in $V(I)$ given in Corollary~\ref{Cor:GBConsequences}. In the following, we are always referring to the multiplicity of a point according to the defining ideal (not with respect to the radical). In other words, our varieties are not reduced and we are essentially considering the closed points of the affine schemes associated to our ideals. 

\smallskip
\begin{theorem}~\label{Thm:autoLowerBound}
   Suppose that $n >2$. The multiplicity of $\bfz$ in $V(I)$ is at least $2n$. Consequently, $V(I)$ contains at least $1$ and at most $n!-2n+1$ points for $n>1$.
\end{theorem}
\begin{proof}
    Given a fixed generic potential $V' = (v'_1,\dots, v'_n)\in \mathbb{C}^n$, consider the family of ideals of spectral invariants $I_{tV'}$ parameterized by $t\in \mathbb{C}$, where $tV'= (tv_1 ',\dots, tv_n ')$. By generic we mean, $V'$ such that $v'_i \neq v'_j$ if $i\neq j$ for any $i,j \in [n]$. The generators of each $I_{tV'}$ are obtained in the same way as $I (= I_\bfz)$ (i.e., taking the nonzero coefficients of the difference $D_V(\lambda) - D_{tV'}(\lambda)$ as a polynomial in $\lambda$); in particular, if $h_k$ is the $k$th spectral invariant of $I_{tV'}$, then $h_k(v,t) = p_k(v_1,\dots,v_n) - p_k(tv'_1,\dots, tv'_n)$, where the $p_k$ are as in ~\eqref{eq:sysSpecInvars}.

    View $t$ as an indeterminate, let $I_X$ denote the ideal of $\CC[v_1,..,v_n,t]$ generated by the $h_i(v,t)$ and $X := V(I_X)  \subseteq \CC^n \times \CC$ the associated affine variety. Notice that, if $g'_k := -\sum_{j=1}^k H(k, k-j)(-1)^jh_j$, then $G' = \{ g'_1,\dots, g'_n\}$ is a Gr\"obner basis of $I_X$. This follows from the same proof as Theorem~\ref{grobnerbasis} via grevlex monomial order with $v_1 > \dots > v_n > t$. The only notable difference is in the corresponding version of Theorem~\ref{homogeneous}, we'd instead find that $T_{\max}(g_k) = H(k,k) + tr$ for some homogeneous $r \in \CC[v_k,\dots,v_n,t]$ of degree $k-1$, but this leaves Corollary~\ref{leadingterm} unchanged. Finally, note that the set of monomials not in $\langle LT(I_X) \rangle$ is $B_X = \left\{t^a v_1^{j_1} v_2^{j_2} \cdots v_n^{j_n} | 0 \leq j_\ell < \ell, 0 \leq a \right\}$. By the arguments of Theorem~\ref{hilbert}, we have find that $HP_{\CC[v,t]/I_X}(s) = (n!)s -\frac{n! (n(n-1) - 4)}{4}$, and so $X$ is $1$-dimensional. 
    
     Consider the hyperplane $\mathcal{H}:=\mathbb{C}^n\times \{0\}$, we can identify the point $(\textbf{0},0)$ with the potential $\textbf{0}\in V(I_0)$ since the multiplicity of $(\textbf{0},0)$ in the intersection $X \cap\mathcal{H}$ is equal to the multiplicity of $\textbf{0}$ in $ V(I_0)$. By \eqref{dihedral}, for each $\sigma \in D_{n}$ and each $t\in \CC$, we have $(tv_{\sigma(1)}',\dots,tv_{\sigma(n)}',t) \in X$. Thus, $X$ contains at least $2n$ distinct lines through the origin $(\textbf{0},0)\in \mathbb{C}^n\times \mathbb{C}$, which we denote by $L_1, \dots, L_{2n}$. Each $L_i$ is an irreducible component of $X$ and as multiplicity is additive over irreducible components \footnote{E.g. see~\cite[Section 3]{fulton1989algebraic}.}, the multiplicity of $(\textbf{0},0)$ in $X$ is at least 2n. As $(\textbf{0},0) \in \mathcal{H}$, it has at least multiplicity 1 in $\mathcal{H}$. Since the multiplicity of a point on an intersection is at least the product of the respective multiplicities, it follows that the multiplicity of $(\textbf{0},0)$ in the intersection $X\cap \mathcal{H}$ is at least $1*2n$.
\end{proof}
\smallskip
\begin{remark}~\label{Rm:TrivExtend2}
    The observation made in Theorem~\ref{Thm:autoLowerBound} trivially extends more generally to the spectral invariants of graphs equipped with a fixed potential $V$. In particular, given a graph $G$ with $n$ vertices, there are at most $n! - |aut(G)|$ nontrivial potentials that are Floquet isospectral to $V = (v_1,\dots, v_n)$, where $aut(G)$ is the automorphism group of $G$. As $V$ is general, a nontrivial potential Floquet isospectral to $V$ is one that does not arise from permuting the coordinates of $V$ by the symmetries of $G$.
\end{remark}
\section{Symmetries and Specialized Potentials}$~$ ~\label{Sec:4}
Given an $n\mathbb{Z}$-periodic potential $V=(v_1,\dots,v_n)$, let $-V=(-v_1,\dots,-v_n)$ and let $\overline{V}=(\overline{v_1},\dots,\overline{v_n})$. 

\begin{lemma}~\label{negative}
    If $V$ is isospectral to $\bfz$, then $-V$ is isospectral to $\bfz$. 
\end{lemma}

\begin{proof}
    Follows trivially as replacing $V$ with $-V$ at most changes the sign of the $p_k$ by Lemma~\ref{parity}.
\end{proof}
\smallskip
\begin{lemma}~\label{conjugate}
    If $V$ is isospectral to $\bfz$, then $\overline{V}$ is isospectral to $\bfz$.
\end{lemma}

\begin{proof}
As the coefficients of the $p_k$ are all real we have that $p_k(\overline{V}) = \overline{p_k(V)}$.
\end{proof}
\smallskip

\begin{corollary}~\label{symmetries}
    If $V$ is isospectral to $\bfz$, then any potential $V'$ that is a dihedral permutation of $V$, $-V$, $\overline{V}$, or $-\overline{V}$ is isospectral to $\bfz$.
\end{corollary}

\begin{proof}
    This follows immediately from ~\eqref{dihedral}, Lemma~\ref{negative}, and Lemma~\ref{conjugate}.
\end{proof}

\begin{remark}\label{rm:symsVI} Corollary~\ref{symmetries} together with ~\eqref{dihedral} give us that for every potential $V$ that is isospectral to $\bfz$, we can immediately find up to $8n-1$ other potential functions by taking dihedral permutations of $V$ and their conjugates, negations, and conjugate-negations.
\end{remark}

We next turn our attention to a particular specialization of the potential. Consider a potential function $U$ for the $n\mathbb{Z}$-periodic square lattice such that $u_j = -u_{n-j+1}$. For even $n=2m$, this yields a potential function of the form $(u_1, \ldots, u_m, -u_m, \ldots, -u_1)$. For odd $n= 2m+1$, this yields a potential function of the form $(u_1, \ldots, u_m, 0, -u_m, \ldots, -u_1)$.
\smallskip
\begin{lemma}~\label{negationequivalence}
    $[\lambda^k]\zeta_U = [\lambda^k]\zeta_{-U}$ for $0 \leq k \leq n$.
\end{lemma}

\begin{proof}
    By ~\eqref{dihedral}, each $[\lambda^k]\zeta_U$ exhibits reflective symmetry; $-U$ is the reflection of $U$.
\end{proof}
\smallskip
\begin{corollary}~\label{cancellation}
    $[\lambda^{n-2k-1}]\zeta_U = 0$ for all $0 \leq k \leq m$.
\end{corollary}

\begin{proof}
    By Lemma~\ref{parity}, $[\lambda^{n-2k-1}]\zeta_U$ is entirely composed of odd degree terms. Therefore, $[\lambda^{n-2k-1}]\zeta_U = -[\lambda^{n-2k-1}]\zeta_{-U}$. However, by Lemma~\ref{negationequivalence}, $[\lambda^{n-2k-1}]\zeta_U = -[\lambda^{n-2k-1}]\zeta_{-U} = -[\lambda^{n-2k-1}]\zeta_U$. Therefore, $[\lambda^{n-2k-1}]\zeta_U = 0$.
\end{proof}

Since every other $\lambda$ coefficient is $0$ under this specialization, this reduces the system from $n$ equations in $n$ variables to $m$ equations in $m$ variables. 
We now turn our attention to this specialized system.

 Define \[e_k' := \sum_{1 \leq j_1 < \cdots < j_k \leq m} \prod_{\ell = 1}^ku_{j_\ell}^2 \text{ and } H'(a, b) := \sum_{a \leq j_1 \leq \cdots \leq j_b \leq m} \prod_{\ell = 1}^bu_{j_\ell}^2,\] the elementary symmetric polynomials of degree $k$ in $\CC[u]$ and the complete homogeneous symmetric polynomials of degree $b$ on $u_a,\dots, u_m$, both under the map $u_i \to u_i^2$ for all $i$. Also let
 \[p_k' := (-1)^{n+k} \zeta_U\left[\lambda^{n-2k}\right], I' := \langle p'_1, \ldots, p'_m \rangle, g'_k := \sum_{j = 1}^k H'(k, k-j)p'_j, \text{ and } G' = \{g'_1, \ldots, g'_m\}.\]

These definitions are entirely analogous to those used in Section~\ref{Sec:3}. For example, in this case, we have the following analog of~\eqref{zetalambdak},
    \begin{equation}\label{eq:p'leading} T_\text{max}\left(p'_{k}\right) = (-1)^k[x^{2k}] \prod_{i-1}^{m}(1+u_ix)(1-u_ix) = e'_{k}.\end{equation}

\begin{remark}\label{rm:symsVI'} $V(I')$ naturally inherits the symmetries of Lemma~\ref{negative} and Lemma~\ref{conjugate}.
\end{remark}
\smallskip
\begin{theorem}~\label{thm:hilbetalspec}
    $G'$ is a Gröbner basis of $I'$, $\langle LT(I') \rangle = \langle u_1^2, u_2^4, \ldots, u_m^{2m} \rangle$ in grevlex monomial ordering, and the affine Hilbert polynomial of $I'$ is $HP_{R/I'}(s) = 2^mm!$. 
\end{theorem} 
\begin{proof}
    Since the highest degree terms of each $p'_k$ are the $k$th elementary polynomial in $u_1^2, \ldots, u_m^2$, these results arise via arguments analogous to those given in Section~\ref{Sec:3} for the unspecialized potential. The Hilbert polynomial comes from the replacement of $B$ in the proof of Theorem~\ref{hilbert} with $B' = \left\{u_1^{j_1}u_2^{j_2} \cdots u_m^{j_m} : 0 \leq j_\ell < 2\ell \right\}$.
\end{proof}
\smallskip
\begin{corollary}~\label{Cor:GBConsequencesspec}  The dimension of $V(I')$ is $0$, and the degree of $V(I')$ is $2^mm!$.
Thus, $V(I)$ has $2^mm!$ points counting multiplicity, and is composed of between $1$ and $2^mm!$ points. 
\end{corollary}

Given the properties of this specialization along with the fact that there are known nonzero points in $V(I')$ when $n$ is divisible by 4~\cite{flmrp}, we propose the following conjecture.

\begin{conjecture}~\label{conj:1}
   There exist nonzero $q\ZZ$-periodic potentials $V$ of the form $(v_1,v_2,\dots,-v_2, -v_1)$ for $q > 3$ such that $\Delta$ and $\Delta + V$ are Floquet isospectral.
\end{conjecture}

Just as the even potentials found in ~\cite{flmrp} support Conjecture~\ref{conj:1}, they also suggest the following conjecture.
\smallskip
\begin{conjecture}~\label{conj:2}
    There exist nonzero $q\ZZ$-periodic potentials $V$ of the form $(v_1, v_2,\dots, \overline{v_2}, \overline{v_1})$ for $q > 3$ such that $\Delta$ and $\Delta + V$ are Floquet isospectral.
\end{conjecture}

For the sake of completeness we offer a short proof that the potentials found in~\cite{flmrp} affirm Conjecture~\ref{conj:1} in the case $4|q$ and Conjecture~\ref{conj:2} in the case $q$ is even.

\begin{proof}[Proof of Conjecture ~\ref{conj:1} when $4|q$ and Conjecture~\ref{conj:2} for even $q >3$.]

Let $q = 2m$ for $m>1$. From ~\cite[Theorem 3.1]{flmrp} we have that the $q\ZZ$-periodic potential $V$ such that $V(1) = 1+i$, $V(2) = 1-i$, $V(m+1) = -1 + i$, $V(m+2) = -1 -i$ and $V(j) = 0$ for $j$ on all other vertex orbits is Floquet isospectral to $\bfz$.  

Applying the cycle permutation $(1 \ 2 \cdots q)$ $m-1$ times to $V$ gives us a $q\ZZ$-periodic potential $V'$ such that $V'(1) = -1-i$, $V'(m) = 1+i$, $V'(m+1) = 1-i$, $V'(q) = -1 -i$, and $V'(j) = 0$ otherwise, which is isospectral to $\bfz$ by ~\eqref{dihedral}. This proves Conjecture~\ref{conj:2} for even $q>3$.

Now suppose $m = 2k$. Applying the cycle permutation $(1\ 2 \cdots q)$ $k-1$ times to $V$ gives us a $q\ZZ$-periodic potential $V'$ such that $V'(k) = -1-i$, $V'(k+1) = 1+i$, $V'(m+k) = 1-i$, $V'(m+k+1) = -1 -i$, and $V'(j) = 0$ otherwise, which is isospectral to $\bfz$ by ~\eqref{dihedral}. This proves Conjecture~\ref{conj:1} for $q>3$ such that $4 | q$.
\end{proof}

Next, we will discuss some experimental findings that support these conjectures. 
\section{Experimental Results}~\label{Sec:5}

In this section, we discuss the values of $V(I)$ that we found experimentally for small values of $n$. For the results of this section, we performed symbolic and numerical computations using Macaulay2~\cite{M2} and Bertini~\cite{Bertini}. The numerical computations utilized the package \cite[NumericalAlgebraicGeometry]{NumericalAlgebraicGeometrySource, NumericalAlgebraicGeometryArticle} and the solutions to ~\eqref{eq:sysSpecInvars} were certified through the package \cite[NumericalCertification]{NumericalCertificationSource}, which utilizes alpha theory to prove these numerical solutions correspond to true solutions~\cite{alphaCertified-paper}. In what follows, unless otherwise specified, solutions refer to nonzero points of $V(I)$ or $V(I')$, symmetries with respect to $V(I)$ refer to Remark~\ref{rm:symsVI}, and symmetries with respect to $V(I')$ refer to Remark~\ref{rm:symsVI'}.

\begin{table}
\centering
\begin{tabular}{ | m{.2cm} | m{1cm}| m{1cm} |m{1.1cm} |m{1.4cm} |m{1.5cm} | m{2cm} |m{2.2cm} | m{2.3cm} | } 
\hline
 $n$ & mult. at $\bfz$ & mult. of points $\neq \bfz$. & unique points & unique points mod $D_n$ & unique points mod all symmetries & singular points mod $D_n$ & mult.:counts & Conjecture~\ref{conj:2} unique points.\\ [0.5ex] 
 \hline\hline
  3 & 6 & 0  & 1  & 1 & 1 & 1 & \{6:1\} & 0  \\ 
  \hline
  4 & 16 & 8  & 9  & 2 & 1 & 1 &\{1:8, 16:1\} & 4 \\ 
  \hline
   5 & 60 & 60 & 61 & 7 & 3 & 1 &\{1:60, 60:1 \} & 4\\ 
  \hline
   6 & 120 & 600 & 529 & 45 & 20 & 3 &\{ 1:504 , 2:12, 6:12, 120:1\} & 32\\ 
  \hline
  7 & 588 & 4452 & 4453  & 319 & 105 &1 & \{1:4452, 588:1 \} & 120\\ 
  \hline
\end{tabular}
\caption{Computational findings for $V(I)$.}\label{tab:VI}
\end{table}

Our computational findings related to the elements of $V(I)$ and $V(I')$ are summarized in Tables~\ref{tab:VI} and ~\ref{tab:VI'}, respectively. The multiplicity at $\bfz$ was calculated by symbolically solving for the degree of the tangent cone ideals of $V(I)$ and $V(I')$, respectively (with the exception of $n=7$ for $V(I)$).  The second to last column is formatted as $i$:$j$ to indicate that there are $j$ multiplicity $i$ points in the respective variety\footnote{Note that the degrees of the singularities other than $\bfz$ were not certified and instead are based on output from Bertini. }. The remainder of the entries were computed and certified numerically.

\begin{table}
\centering
\begin{tabular}{ | m{.3cm} | m{1cm}| m{1cm} |m{1.1cm} |m{1.5cm} | m{2cm} |m{3.2cm} | } 
\hline
 $n$ & mult. at $\bfz$ & mult. of points $\neq \bfz$. & unique points & unique points mod symmetries & singular points mod symmetries & mult.:counts\\ [0.5ex] 
 \hline\hline
  4 & 4 & 4  & 5  & 2 & 1 &\{1:4, 4:1\} \\ \hline
  5 & 4 & 4  & 5  & 2 & 1 &\{1:4, 4:1\} \\\hline
  6 & 8 & 40  & 41  & 11 & 1 &\{1:40, 8:1\} \\ \hline
  7 & 8 & 40  & 41  & 11 & 1 &\{1:40, 8:1\} \\\hline
  8 & 16 & 368  & 357 & 90 & 2 &\{1:352, 4:4, 16:1\} \\ \hline
  9 & 16 & 368  & 369 & 93 & 1 &\{1:368 16:1\} \\ \hline
  10 & 32 & 3808  & 3781 & 946 & 2 &\{1:3776, 8:4 32:1\} \\ \hline
  11 & 32 & 3808  & 3809 & 953 & 1 &\{1:3808, 32:1\} \\  \hline
\end{tabular}
\caption{Computational findings for $V(I')$.}\label{tab:VI'}
\end{table}

The one nonzero solution up to symmetry when $n=4$ is exactly that found in~\cite{flmrp}.  In Figure~\ref{5Z}, we plot all nonzero solutions up to dihedral symmetries for $n=5$. In Figure~\ref{6Z}, we show the nonzero singular solutions for $n=6$, the left-hand and right-hand images correspond to the points of $V(I)$ of multiplicities $2$ and $6$, respectively.  Figure~\ref{7Z} displays all nonzero points of $V(I')$ for $n=7$ up to symmetries.  Figure~\ref{9Z0} is a rendering of all nonzero solutions of $V(I)$ for $n=9$ where at least two consecutive potential values are $0$; the nonzero values of the leftmost solution are solutions to $z^6+1=0$.

\begin{figure}[h]
    \centering
    \includegraphics[width=.9\linewidth]{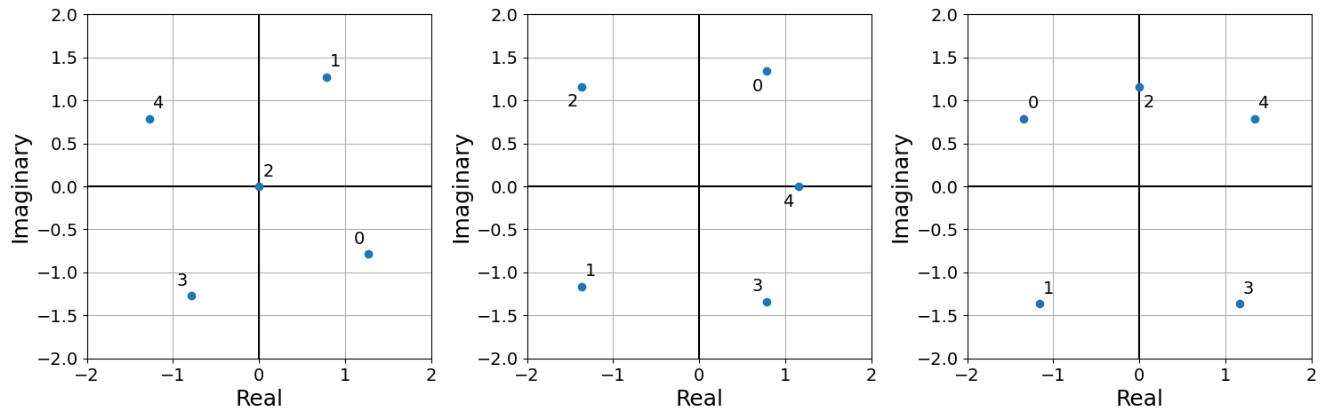}
    \caption{All solutions for $n=5$ up to symmetries.}
    \label{5Z}
\end{figure}

\begin{figure}[h]
    \centering
    \includegraphics[width=0.7\linewidth]{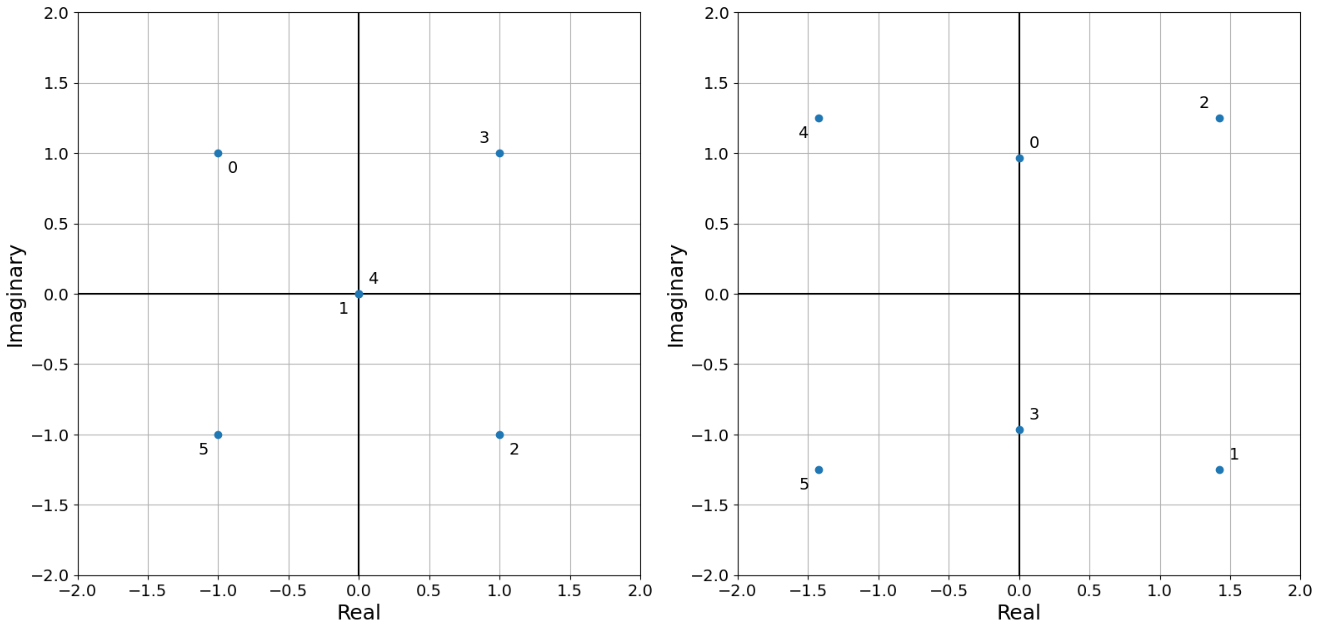}
    \caption{The nonzero singular solutions for $n=6$.}
    \label{6Z}
\end{figure}

\begin{figure}[h]
    \centering
    \includegraphics[width=.9\linewidth]{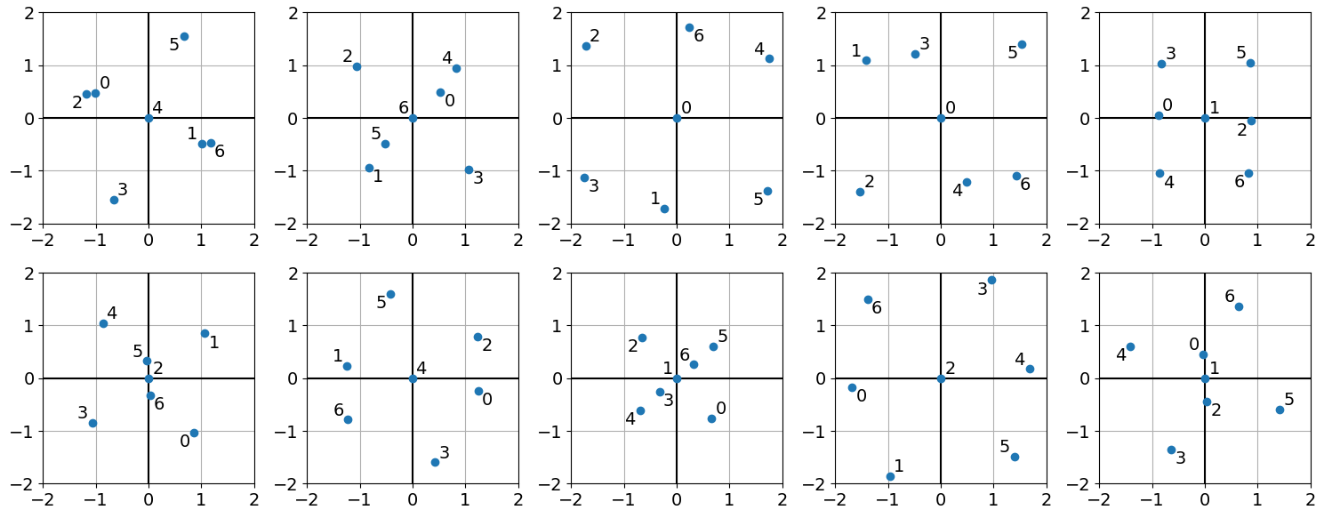}
    \caption{All specialized $n=7$ solutions up to symmetries.} 
    \label{7Z}
\end{figure}

\begin{figure}[h]
    \centering
    \includegraphics[width=.9\linewidth]{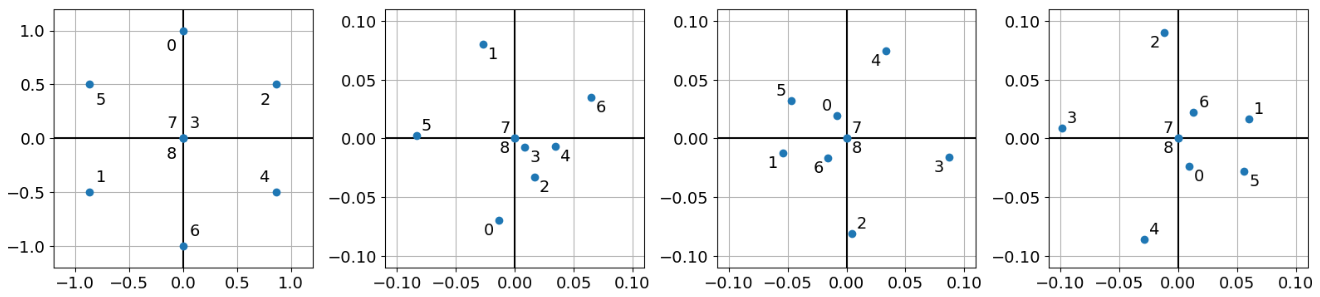}
    \caption{All solutions for $n=9$ with two zeroes in a row up to symmetry.}
    \label{9Z0}
\end{figure}

\begin{figure}[h]
    \centering
    \includegraphics[width=0.7\linewidth]{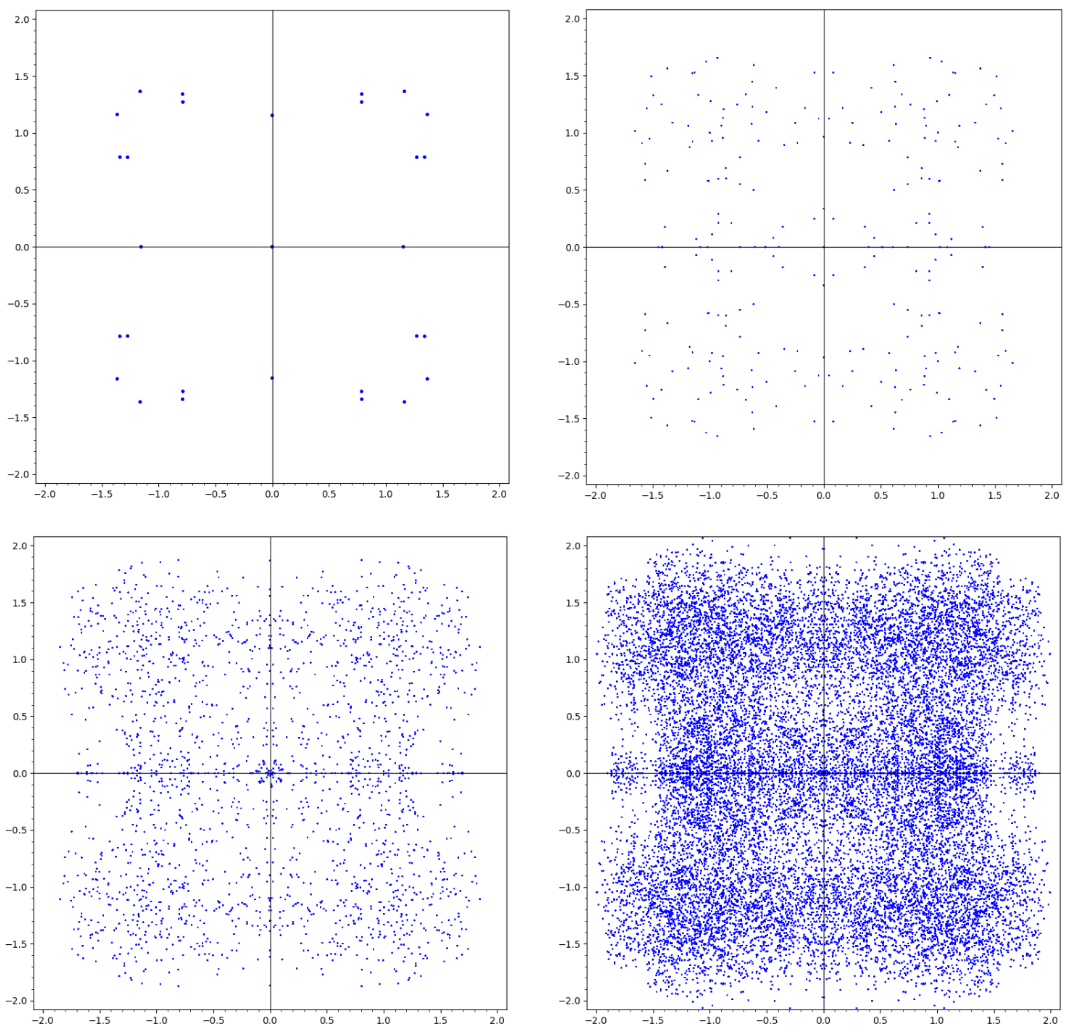}
    \caption{All solutions for $n=5$, $n=6$, $n=7$, and $n=8$.} 
    \label{all}
\end{figure}
In Figure ~\ref{all}, we plot all the values found that the potentials Floquet isospectral to $\bfz$ can take together in the complex plane for $n=5$ up to $n = 8$.  The code used to obtain the data in these tables is available on GitHub\footnote{mattfaust.github.io/FloIsoDL/flz.m2}.

Based on the data of Table~\ref{tab:VI'}, we propose the following refinement of Conjecture~\ref{conj:1}.
\begin{conjecture}~\label{conj:3}
When $q$ is odd, then $V(I')$ has only regular nonzero points, and when it is an even number greater than $6$ there is a single singular point, up to symmetry, of multiplicity $2^{\frac{q}{2}-2}$. Moreover, the multiplicity at $\bfz$ is $2^{\lfloor \frac{q}{2}\rfloor}$.
\end{conjecture}
\section{The Case of General Full Rank Sublattices}~\label{Sec:6}
Although so far we have restricted ourselves to the case of $\Gamma$-periodic potentials for $\Gamma = q_1\ZZ~\oplus~q_2\ZZ~\oplus\dots\oplus q_d\ZZ$ and $q_i \in \ZZ^{+}$, the question is also of interest for more general potentials, that is potentials that are $\Gamma$-periodic for $\Gamma = \bq_1\ZZ \oplus \dots \oplus \bq_d \ZZ$ for $\bq_i \in \ZZ^d$, such that $\Gamma$ is a full rank sublattice of $\ZZ^d$. In this setting, the Floquet matrix can be obtained through an identical process as that outlined in Section~\ref{Sec:Intro}, the only difference being that we look at the spectrum of $\Delta+V$ acting on functions $f:\ZZ^d \to \CC$ with Floquet-Bloch boundary conditions
    $f(n+\bq_j) = e^{2 \pi ik_j} f(n)$ with $n\in \ZZ^d, k_j\in[0,1),$ and $j = 1,2,\dots, d.$ If the $\bq_i = c_i\mathbf{e}_i$ for some integers $c_i$, then we have returned to the setting of Section~\ref{Sec:Intro}.

Notice that when $\Gamma = q_1\ZZ~\oplus~q_2\ZZ~\oplus\dots\oplus q_d\ZZ$ for $q_i \in \ZZ_{+}$, by~\cite{CFK}, when $\prod_{i=1}^d q_i > 3^d$, then, regardless of what $\Gamma$ is, as one $q_i$ must be greater than $3$, there will exist nonzero $\Gamma$-periodic potentials Floquet isospectral to $\bfz$. In the more general setting where the $q_i$ are allowed to be vectors in $\ZZ^d$, there is no similar bound that once exceeded implies the existence of nonzero potentials Floquet isospectral to $\bfz$. We now demonstrate this. 

\begin{theorem}
    For every $d>1$ and $n \in \ZZ_+$, there exists a full-rank sublattice $\Gamma$ of $\ZZ^d$ with $|\ZZ^d / \Gamma| > n$ such that the only $\Gamma$-periodic potentials Floquet isospectral to $\bfz$ is $\bfz$.
\end{theorem}
\begin{proof}
We prove the claim by giving explicit periods with this property.

For $d=2$, consider a potential that is periodic with respect to $\Gamma = \bq_1 \ZZ \oplus \bq_2 \ZZ$, where $\bq_1 = (a,0)$ for some integer $a>2$ and $\bq_2 = (1,1)$\footnote{To our knowledge, this family first appeared in ~\cite{fk2} in relation to their exceptional band edge behavior.}. It is easy to show that, in this case, we get a matrix of the form: \[         L_{V}=\left(\begin{smallmatrix}
            V(1) & 1+z_2^{-1} & 0 & \dots & 0 & z_1^{-1}(1+z_2) \\
            1+z_2 & V(2) & 1+z_2^{-1} & 0 & \dots & 0 \\
            0 & 1+z_2 & \ddots & \ddots & \ddots & \vdots \\
            \vdots & \ddots & \ddots & \ddots & \ddots & 0\\ 
            0 & 0 & \dots & 1+z_2 & V(a-1) & 1+z_2^{-1} \\
            z_1(1+z_2^{-1}) & 0 & \dots & 0 & 1+z_2 & V(a) \\
        \end{smallmatrix}\right).\]
Noticing that we get a diagonal matrix after setting $z_2 = -1$, we see that the only potential of this period that is Floquet isospectral to $\bfz$, is $\bfz$ (as the ideal of spectral invariants is just the first $a$ elementary symmetric polynomials). More generally, this implies that, given any potential periodic with respect to this lattice, the potentials Floquet isospectral to it must be its coordinate-wise permutations; that is, if $V'$ is Floquet isospectral to $V$, then there is a $\sigma \in S_a$ such that $V'(i) = V(\sigma(i))$. The same is true when $a=2$, albeit the Floquet matrix has a slightly different form. 

For $d > 2$, consider a $\Gamma$-periodic potential for $\Gamma = \bq_1\ZZ\oplus \bq_2\ZZ\oplus\dots\oplus \bq_d\ZZ$ with $\bq_1 = (a,0,0,\dots, 0),\bq_2= (1,1,0,\dots,0),$ $\bq_3 = (1,0,1,0,\dots,0), \dots,$ and $\bq_d= (1,0,\dots, 0,1)$, where $a \in \ZZ^+$. The matrix will have a similar form, just replace every $1+z_2$ with $1+z_2 + z_3 + \dots + z_d$ and every $1+z_2^{-1}$ with $1+z_2^{-1} + z_3^{-1} + \dots + z_d^{-1}$. It follows that solving for values of $z_2,\dots, z_d$ such that $z_2 +\dots + z_d = z_2^{-1} + \dots + z_d^{-1} = -1$ leads to the same observation as the $d=2$ case.
\end{proof}

Notice that when $\Gamma = q_1\ZZ~\oplus~q_2\ZZ~\oplus\dots\oplus q_d\ZZ$ for $q_i \in \ZZ^{+}$, a nonzero $q_1\ZZ \oplus \ZZ \oplus \dots \oplus \ZZ$-periodic potential Floquet isospectral to $\bfz$ extends to a nonzero $\Gamma$-periodic potential Floquet isospectral to $\bfz$ (as $\Gamma$ is a subgroup of $q_1\ZZ \oplus \ZZ \oplus \dots \oplus \ZZ$). See~\cite[Section 4.2]{CFK} for more details.
When $\Gamma = (a,0) \ZZ \oplus (1,1) \ZZ$ for some integer $a>1$, $\Gamma$ is not a subgroup of $b\ZZ \oplus \ZZ$ or $\ZZ \oplus b\ZZ$ for any integer $b > 1$. This leads us to the following problem. \smallskip

\begin{problem}~\label{prob:2}
    For the $\ZZ^d$-periodic grid graph, classify all periods $\Gamma$ such that there exist nonzero $\Gamma$-periodic potentials that are Floquet isospectral to the discrete Laplacian. 
\end{problem}

Problem~\ref{prob:2} has been solved in the case where $d=1$~\cite{CFK}, but it is still open for $d \geq 2$. Even in the case where $\Gamma = \bq_1\ZZ \oplus \dots \oplus \bq_d \ZZ$ for $\bq_i = c_ie_i$ for $c_i \in \ZZ$, Problem~\ref{prob:2} is open for $d >2$; the barrier being the case when all $c_i$ are $3$. 
In these two versions of Problem~\ref{prob:2}, the corresponding ideals of spectral invariants generally contain more generators than just perturbations of elementary symmetric polynomials. 

We briefly comment on the general case when $d=2$. The only full rank sub-lattices that are not subgroups of $q \ZZ \oplus \ZZ$ or $\ZZ \oplus q \ZZ$ for all $q > 3$ are those with the following Hermite normal form: $(a,0)\ZZ \oplus (b,c)\ZZ$ for $a,b,c$ such that $0 \leq b < a$ and $\gcd(a,b),c \in \{1,2,3\}$. We symbolically checked all possible combinations of $a,b,c$ where $ac\leq 8$, and found no examples with nonzero potentials isospectral to $\bfz$.

\section*{Acknowledgments}
This research was partially conducted as part of Summer Undergraduate Research Institute in Experimental Mathematics (SURIEM), and was partially supported by NSF DMS-2244461 and by Michigan State University. The first named author was partially supported by NSF DMS-2052519. The authors thank Dr. Robert Bell for organizing SURIEM. Faust thanks Wencai Liu for introducing him to this problem.

\bibliographystyle{abbrv} 
	\bibliography{main}
	
\end{document}